\documentclass{amsart}
\usepackage{amssymb,amsmath,amsthm}

\newtheorem{thm}{Theorem}[section]
\newtheorem{lem}[thm]{Lemma}
\newtheorem{prop}[thm]{Proposition}
\newtheorem{cor}[thm]{Corollary}
\theoremstyle{definition}

\newtheorem{rem}[thm]{Remark}
\numberwithin{equation}{section}

\newcommand{\N}{\mathbb{N}}

\newcommand{\C}{\mathbb{C}}
\newcommand{\bk}{\mathbf{k}}
\newcommand{\bm}{\mathbf{m}}
\newcommand{\bp}{\mathbf{p}}
\newcommand{\bq}{\mathbf{q}}
\newcommand{\Sh}{\mathrm{Sh}}
\newcommand{\lra}{\longrightarrow}

\author{Shuji Yamamoto}
\address{%
Department of Mathematics, Faculty of Science and Technology, 
Keio University\\
3-14-1 Hiyoshi, Kohoku-ku, Yokohama, 223-8522, JAPAN}
\email{yamashu@math.keio.ac.jp}
\thanks{This work was supported in part by Grant-in-Aid for JSPS Fellows (No.\ 09J05093) and by JSPS Grant-in-Aid for Young Scientists (S) (No.\ 21674001). }

\title{A sum formula of multiple $L$-values}

\begin{document}

\begin{abstract}
We prove an alternating sum formula of certain multiple $L$-values 
conjectured by Essouabri, Matsumoto and Tsumura, 
which generalizes the sum formula of multiple zeta values. 
The proof relies on the method of partial fraction decomposition. 
\end{abstract}

\maketitle

\section{Introduction}
The multiple zeta values 
\[\zeta(k_1,\ldots,k_n)=\sum_{m_1,\ldots,m_n=1}^\infty 
\frac{1}{m_1^{k_1}(m_1+m_2)^{k_2}\cdots(m_1+\cdots+m_n)^{k_n}}\]
and identities among them have been studied by many authors. 
A basic example of the linear relations among multiple zeta values 
is the sum formula 
\begin{equation}\label{eq:SumFormula}
\sum_{\substack{k_1,\ldots,k_{n-1}\geq 1,k_n\geq 2\\ k_1+\cdots+k_n=k}}
\zeta(k_1,\ldots,k_n)=\zeta(k)
\end{equation}
which holds for any integers $k>n>0$. 
The case of $n=2$, i.e., the identity 
\[\zeta(1,k-1)+\zeta(2,k-2)+\cdots+\zeta(k-2,2)=\zeta(k)\]
goes back to Euler, 
and the general case was proven independently by 
Granville \cite{G97} and Zagier (unpublished). 

Note that the identity \eqref{eq:SumFormula} can be viewed as 
a decomposition of the Riemann zeta value $\zeta(k)$ 
into a finite sum of multiple zeta values. 
Then a natural question arises: 
Are there similar decomposition formulas 
for Dirichlet $L$-values 
$L(k,\chi)=\sum_{m=1}^\infty\frac{\chi(m)}{m^k}$, 
or polylogarithm functions 
$\mathit{Li}_k(t)=\sum_{m=1}^\infty\frac{t^m}{m^k}$ ?

In \cite{EMT10}, Essouabri, Matsumoto and Tsumura found 
such formulas for $L$-values (or, in fact, fairly general Dirichlet series) 
when $n\leq 3$, and made a conjecture for general $n$. 
The main purpose of this paper is to prove their conjecture. 

To state the results, we introduce some notation. 
Let $\N=\{1,2,\ldots\}$ be the set of positive integers. 
We take and fix $n,k\in\N$ such that $k>n$, and put 
\[I(r)=I_{k,n}(r)=\biggl\{\bk=(k_1,\ldots,k_n)\in\N^n\biggm|
\begin{array}{c}k_n\geq 2,\ k_1+\cdots+k_n=k,\\ 
k_1=\cdots=k_{r-1}=1
\end{array}\biggr\}\]
for $r=1,2,\ldots,n,n+1$. 
Thus we have 
\[I(1)\supset I(2)\supset \cdots\supset I(n)\supset I(n+1)=\emptyset, \]
and $I(1)$ is the set of all admissible indices for 
the multiple zeta values $\zeta(\bk)$ of weight $k$ and depth $n$. 

We consider the multiple $L$-values of the following type: 
\[L(\bk,f,J)=\sum_{m_1,\ldots,m_n=1}^\infty 
\frac{f\bigl(\sum_{j\in J}m_j\bigr)}
{m_1^{k_1}(m_1+m_2)^{k_2}\cdots(m_1+\cdots+m_n)^{k_n}}.\]
Here $f\colon\N\lra\C$ is a function and 
$J$ is a nonempty subset of $\{1,\ldots,n\}$. 
Observe that, when $f$ is periodic and $J$ is of the form 
$\{j\}$ or $\{1,\ldots,j\}$, then these $L(\bk,f,J)$ may be 
viewed as special cases of the multiple $L$-values introduced 
by Arakawa-Kaneko \cite{AK04}. 

Our main result is the following: 

\begin{thm}\label{thm:MainThm1}
The identity 
\begin{equation}\label{eq:MainThm1}
\sum_{\emptyset\ne J\subset\{1,\ldots,n\}}
\sum_{\bk\in I(\max J)}
(-1)^{\lvert J\rvert-1}L(\bk,f,J)=\sum_{m=1}^\infty\frac{f(m)}{m^k} 
\end{equation}
holds. Here $f\colon\N\lra\C$ is a function such that 
$L(\bk,f,J)$ are absolutely convergent 
for all $J$ and all $\bk\in I(\max J)$. 
\end{thm}

A sufficient condition for absolute convergence is that 
$f(m)=O(m^{k-n-\varepsilon})$ for some $\varepsilon>0$. 
In fact, if we put $r=\max J$, then any $\bk=(k_1,\ldots,k_n)\in I(r)$ 
satisfies $(k_r-1)+\cdots+(k_n-1)=k-n$. Hence we have 
\begin{align*}
&\sum_{m_1,\ldots,m_n=1}^\infty 
\frac{\bigl|f\bigl(\sum_{j\in J}m_j\bigr)\bigr|}
{m_1^{k_1}\cdots(m_1+\cdots+m_n)^{k_n}}\\
&\leq C\sum_{m_1,\ldots,m_n=1}^\infty 
\frac{\bigl(\sum_{j\in J}m_j\bigr)^{k-n-\varepsilon}}
{m_1^{k_1}\cdots(m_1+\cdots+m_n)^{k_n}}\\
&\leq C\sum_{m_1,\ldots,m_n=1}^\infty 
\frac{(m_1+\cdots+m_r)^{k_r-1}\cdots(m_1+\cdots+m_n)^{k_n-1-\varepsilon}}
{m_1^{k_1}\cdots(m_1+\cdots+m_n)^{k_n}}\\
&=C\sum_{m_1,\ldots,m_n=1}^\infty 
\frac{1}{m_1\cdots(m_1+\cdots+m_{n-1})(m_1+\cdots+m_n)^{1+\varepsilon}}
\end{align*}
for some constant $C$, ant it is well-known that 
the rightmost side is convergent. 

\smallskip 

As noted above, the formula \eqref{eq:MainThm1} was conjectured 
by Essouabri-Matsumoto-Tsumura \cite{EMT10}, and proven for $n=2$ and $3$. 
As they remarked, when $f(m)=1$ identically, 
\eqref{eq:MainThm1} reduces to the sum formula \eqref{eq:SumFormula} 
for multiple zeta values. 
Our proof of \eqref{eq:MainThm1}, 
which relies essentially on the partial fraction decomposition such as 
$\frac{1}{ab}=\frac{1}{a+b}\bigl(\frac{1}{a}+\frac{1}{b}\bigr)$, 
is similar to that of \eqref{eq:SumFormula} by Granville, 
but makes no use of generating functions. 

\bigbreak 

To obtain a decomposition formula for the polylogarithm $\mathit{Li}_k(t)$, 
we apply Theorem \ref{thm:MainThm1} to the function $f(m)=t^m$, 
where $t$ is any complex number with $\lvert t\rvert\leq 1$. 
Then the left-hand side of \eqref{eq:MainThm1} becomes 
\begin{equation}\label{eq:f(m)=t^m LHS}
\sum_{\emptyset\ne J\subset\{1,\ldots,n\}}
\sum_{\bk\in I(\max J)}
\sum_{m_1,\ldots,m_n=1}^\infty 
\frac{(-1)^{\lvert J\rvert-1}t^{\sum_{j\in J}m_j}}
{m_1^{k_1}\cdots(m_1+\cdots+m_n)^{k_n}}. 
\end{equation}
Put $I'(r)=I(r)\setminus I(r+1)$ ($r=1,\ldots,n$). 
If we fix $\bk\in I'(r)$ and $m_1,\ldots,m_n\geq 1$, 
the corresponding sum of numerators in \eqref{eq:f(m)=t^m LHS} is 
\begin{align*}
\sum_{J\ne\emptyset,\max J\leq r}
(-1)^{\lvert J\rvert-1}t^{\sum_{j\in J}m_j}
&=-\sum_{J\ne\emptyset,\max J\leq r}\prod_{j\in J}(-t^{m_j})\\
&=-\bigl\{(1-t^{m_1})\cdots(1-t^{m_r})-1\bigr\}\\
&=1-(1-t^{m_1})\cdots(1-t^{m_r}). 
\end{align*}
Thus Theorem \ref{thm:MainThm1} implies 
(in fact, is equivalent to) the following formula for the polylogarithm: 
\begin{thm}\label{thm:MainThm2}
For any complex number $t$ such that $\lvert t\rvert\leq 1$, we have 
\begin{equation}\label{eq:MainThm2}
\sum_{r=1}^n\sum_{\bk\in I'(r)}\sum_{m_1,\ldots,m_n=1}^\infty 
\frac{1-(1-t^{m_1})\cdots(1-t^{m_r})}
{m_1^{k_1}\cdots(m_1+\cdots+m_n)^{k_n}}
=\mathit{Li}_k(t). 
\end{equation}
\end{thm}

We can make it more symmetric by taking 
the difference of \eqref{eq:MainThm2} and itself for $t=1$ 
(the latter is just the sum formula \eqref{eq:SumFormula}). 

\begin{cor}\label{cor:CorMainThm2}
We have 
\begin{equation}\label{eq:CorMainThm2}
\sum_{r=1}^n\sum_{\bk\in I'(r)}\sum_{m_1,\ldots,m_n=1}^\infty
\frac{(1-t^{m_1})\cdots (1-t^{m_r})}{m_1^{k_1}\cdots(m_1+\cdots+m_n)^{k_n}}
=\sum_{m=1}^\infty\frac{1-t^m}{m^k}. 
\end{equation}
\end{cor}

\bigbreak 

To prove \eqref{eq:MainThm1}, 
it suffices to consider functions $f$ supported at a single element $m\in\N$. 
To state explicitly, we put 
\begin{gather*}
M(m,J)=\biggl\{\bm=(m_1,\ldots,m_n)\in\N\biggm|
\sum_{j\in J}m_j=m\biggr\},\notag \\
S(m,J)=\sum_{\bk\in I(\max J)}\sum_{\bm\in M(m,J)}
\frac{1}{m_1^{k_1}(m_1+m_2)^{k_2}\cdots(m_1+\cdots+m_n)^{k_n}} 
\end{gather*}
for $m\in\N$ and $\emptyset\ne J\subset\{1,\ldots,n\}$ 
(recall that $n$ and $k$ are fixed). 
Then Theorem \ref{thm:MainThm1} is equivalent to the following:

\begin{prop}\label{prop:MainThm3}
For any $m\in\N$, we have 
\begin{equation}\label{eq:MainThm3}
\sum_{\emptyset\ne J\subset\{1,\ldots,n\}}
(-1)^{\lvert J\rvert-1}S(m,J)=\frac{1}{m^k}. 
\end{equation}
\end{prop}

From the next section, we will prove the identity \eqref{eq:MainThm3} 
by computing the left-hand side in three steps: 
\begin{enumerate}
\item[(i)] For $1\leq l\leq r\leq n$, 
compute the sum over $J$ such that $\lvert J\rvert=l$ and $\max J=r$ 
(Proposition \ref{prop:fix l,r}). 
\item[(ii)] Sum up the values of (i) for $r=l,\ldots,n$ 
(Proposition \ref{prop:fix l}). 
\item[(iii)] Sum up alternatingly the values of (ii) for $l=1,\ldots,n$. 
\end{enumerate}

\section{First step}
First we prepare a lemma. 
Let $N$ be a positive integer and $\underline{s}=(s_0,s_1,\ldots,s_\ell)$ 
a non-decreasing sequence of integers such that $s_0=0$ and $s_\ell=N$. 
Then we denote by $\Sh(\underline{s})$ the set of all shuffles of 
the sequences $(s_{j-1}+1,\ldots,s_j)$, 
i.e.\ all bijections $\sigma\colon\{1,\ldots,N\}\to\{1,\ldots,N\}$ 
which are increasing on each subset $\{s_{j-1}+1,\ldots,s_j\}$ 
for $j=1,\ldots,\ell$. 

\begin{lem}\label{lem:shuffle}
Let $N$ and $\underline{s}=(s_0,s_1,\ldots,s_\ell)$ be as above. 
Then, for indeterminates $x_1,\ldots,x_N$, the identity 
\begin{equation}\label{eq:shuffle}
\sum_{\sigma\in\Sh(\underline{s})}
\prod_{i=1}^N\frac{1}{x_{\sigma^{-1}(1)}+\cdots+x_{\sigma^{-1}(i)}}
=\prod_{j=1}^\ell\prod_{i=1}^{s_j-s_{j-1}}
\frac{1}{x_{s_{j-1}+1}+\cdots+x_{s_{j-1}+i}}
\end{equation}
holds. 
\end{lem}
\begin{proof}
We use induction on $\ell$ and $N$. 
When $\ell=1$ or $N=1$, the claim is trivial. 

In general case, we may assume that $s_{j-1}<s_j$ for all $j=1,\ldots,\ell$ 
(if not, we can reduce $\ell$ by eliminating redundant $s_j$). 
Then the set $\Sh(\underline{s})$ is partitioned into $\ell$ subsets 
\[\Sh(\underline{s})_j
=\bigl\{\sigma\in\Sh(\underline{s})\bigm|\sigma^{-1}(N)=s_j\bigr\} 
\quad (j=1,\ldots,\ell). \]
The induction hypothesis for $N-1$ variables (with $x_{s_j}$ omitted) 
implies that 
\begin{align*}
\sum_{\sigma\in\Sh(\underline{s})_j}
&\prod_{i=1}^N\frac{1}{x_{\sigma^{-1}(1)}+\cdots+x_{\sigma^{-1}(i)}}\\
&=\prod_{\substack{a=1\\ a\ne j}}^\ell
\prod_{i=1}^{s_a-s_{a-1}}\frac{1}{x_{s_{a-1}+1}+\cdots+x_{s_{a-1}+i}}
\times\prod_{i=1}^{s_j-1-s_{j-1}}
\frac{1}{x_{s_{j-1}+1}+\cdots+x_{s_{j-1}+i}}\\ 
&\qquad\times\frac{1}{x_1+\cdots+x_N}\\
&=\frac{x_{s_{j-1}+1}+\cdots+x_{s_j}}{x_1+\cdots+x_N}
\prod_{a=1}^\ell\prod_{i=1}^{s_a-s_{a-1}}
\frac{1}{x_{s_{a-1}+1}+\cdots+x_{s_{a-1}+i}}. 
\end{align*}
By summing up over $j$, we obtain the identity \eqref{eq:shuffle}. 
\end{proof}

\begin{rem}
In the proof of Proposition \ref{prop:fix l,r} below, 
we only need Lemma \ref{lem:shuffle} for $\ell=2$. 
On the other hand, Lemma \ref{lem:shuffle} also includes 
an identity of Littlewood \cite[p.~85]{L50} 
as a special case in which $s_j=j$ (hence $\ell=N$). 
The author thanks the referee for giving him 
the information about Littlewood's identity. 
\end{rem}

Now let us start the computation of the sum in \eqref{eq:MainThm3}. 
Fix $n,k\in\N$ such that $k\geq n+1$, and $m\in\N$. 
For $l,r\in\N$ such that $1\leq l\leq r\leq n$, we put 
\[A_l=\sum_{\substack{(p_1,\ldots,p_{l-1})\in\N^{l-1}\\ p_1+\cdots+p_{l-1}<m}}
\frac{1}{p_1(p_1+p_2)\cdots(p_1+\cdots+p_{l-1})}, \]
\[B_{l,r}(p_l,\ldots,p_{r-1})
=\frac{1}{p_l(p_l+p_{l+1})\cdots(p_l+\cdots+p_{r-1})},\] 
and 
\begin{multline*}
C_{l,r}(p_l,\ldots,p_{n-1},k_r,\ldots,k_n)\\
=\frac{1}{(m+p_l+\cdots+p_{r-1})^{k_r}\cdots(m+p_l+\cdots+p_{n-1})^{k_n}}. 
\end{multline*}
Note that 
\[A_l=\sum_{\substack{(p_1,\ldots,p_{l-1})\in\N^{l-1}\\ p_1+\cdots+p_{l-1}<m}}
B_{1,l}(p_1,\ldots,p_{l-1}). \]

\begin{prop}\label{prop:fix l,r}
\[\sum_{\substack{\lvert J\rvert=l\\ \max J=r}}S(m,J)
=A_l\sum_{\bp=(p_l,\ldots,p_{n-1})\in\N^{n-l}}B_{l,r}(\bp)
\sum_{\bk\in I(r)}C_{l,r}(\bp,\bk). \]
\end{prop}

\begin{rem}
Here the notation $B_{l,r}(\bp)$ for $\bp=(p_l,\ldots,p_{n-1})\in\N^{n-l}$ 
means $B_{l,r}(p_l,\ldots,p_{r-1})$. 
In other words, to evaluate $B_{l,r}$ at $\bp$, 
we simply omit the redundant components $p_r,\ldots,p_{n-1}$. 
We also use similar notation in the following. 
\end{rem}

\begin{proof}
There is a map 
\begin{align*}
\Sh(0,l-1,r-1)&\lra \bigl\{J\subset\{1,\ldots,n\}\bigm|
\lvert J\rvert=l,\,\max J=r\bigr\}\\ 
\sigma&\longmapsto \sigma\bigl(\{1,\ldots,l-1\}\bigr)\cup\{r\}
\end{align*}
which is clearly bijective. 
If $\sigma\in\Sh(0,l-1,r-1)$ corresponds to $J$ under this map, 
each $\bm=(m_1,\ldots,m_n)\in M(m,J)$ can be expressed as 
\[m_j=\begin{cases}
p_{\sigma^{-1}(j)} & (1\leq j\leq r-1),\\
m-(p_1+\cdots+p_{l-1}) & (j=r),\\
p_{j-1} & (j>r)
\end{cases}\]
by a unique $\bp=(p_1,\ldots,p_{n-1})\in\N^{n-1}$ such that 
$p_1+\cdots+p_{l-1}<m$. If this is the case, we have 
\[m_1+m_2+\cdots+m_j=m+p_l+\cdots+p_{j-1}\]
for $j\geq r$. 
Hence we can rewrite the definition of $S(m,J)$ as
\begin{equation}\label{eq:S(m,J) by sigma}
\begin{split}
&S(m,J)\\
&=\sum_{\bk\in I(r)}
\sum_{\substack{\bp\in\N^{n-1}\\ p_1+\cdots+p_{l-1}<m}}
\frac{1}{p_{\sigma^{-1}(1)}\cdots(p_{\sigma^{-1}(1)}
+\cdots+p_{\sigma^{-1}(r-1)})}\\
&\hspace{60pt}\times\frac{1}{(m+p_l+\cdots+p_{r-1})^{k_r}\cdots
(m+p_l+\cdots+p_{n-1})^{k_n}}\\
&=\sum_{\substack{\bp\in\N^{n-1}\\ p_1+\cdots+p_{l-1}<m}}
\frac{1}{p_{\sigma^{-1}(1)}\cdots(p_{\sigma^{-1}(1)}
+\cdots+p_{\sigma^{-1}(r-1)})}
\sum_{\bk\in I(r)}C_{l,r}(\bp,\bk). 
\end{split}
\end{equation}
By Lemma \ref{lem:shuffle}, the equality 
\begin{align*}
\sum_{\sigma\in\Sh(0,l-1,r-1)}&\frac{1}
{p_{\sigma^{-1}(1)}\cdots(p_{\sigma^{-1}(1)}+\cdots+p_{\sigma^{-1}(r-1)})}\\
&=\frac{1}{p_1(p_1+p_2)\cdots(p_1+\cdots+p_{l-1})}
\frac{1}{p_l(p_l+p_{l+1})\cdots(p_l+\cdots+p_{r-1})}\\
&=B_{1,l}(\bp)B_{l,r}(\bp)
\end{align*}
holds for each $\bp=(p_1,\ldots,p_{n-1})\in\N^{n-1}$. 
Therefore, by summing up \eqref{eq:S(m,J) by sigma}, we conclude 
\begin{align*}
\sum_{\substack{\lvert J\rvert=l\\ \max J=r}}S(m,J)
&=\sum_{\substack{\bp=(p_1,\ldots,p_{n-1})\in\N^{n-1}\\ p_1+\cdots+p_{l-1}<m}}
B_{1,l}(\bp)B_{l,r}(\bp)\sum_{\bk\in I(r)}C_{l,r}(\bp,\bk)\\
&=A_l\sum_{\bp=(p_l,\ldots,p_{n-1})\in\N^{n-l}}B_{l,r}(\bp)
\sum_{\bk\in I(r)}C_{l,r}(\bp,\bk) 
\end{align*}
as required. 
\end{proof}

\section{Second step}
The purpose in this section is to compute $\sum_{\lvert J\rvert=l}S(m,J)$ 
for $1\leq l\leq n$. By Proposition \ref{prop:fix l,r}, we have 
\begin{equation}\label{eq:|J|=l}
\sum_{\lvert J\rvert=l}S(m,J)=
A_l\sum_{\bp\in\N^{n-l}}D_l(\bp), 
\end{equation}
where we put
\[D_l(\bp)=\sum_{r=l}^nB_{l,r}(\bp)\sum_{\bk\in I(r)}C_{l,r}(\bp,\bk). \]

\begin{lem}\label{lem:r<n}
For $t=l,\ldots,n-1$, 
\begin{equation}\label{eq:r<n}
\sum_{r=l}^tB_{l,r}(\bp)\sum_{\bk\in I(r)}C_{l,r}(\bp,\bk)
=B_{l,t}(\bp)\sum_{\bk\in I(t)}
\frac{1}{m^{k_t}}C_{l,t+1}(\bp,\bk). 
\end{equation}
\end{lem}
\begin{proof}
We use the induction on $t$. 
When $t=l$, the claim is obvious since 
\[C_{l,l}(\bp,\bk)=\frac{1}{m^{k_l}}C_{l,l+1}(\bp,\bk). \]

Let $t\geq l+1$. Then any element of $I(t-1)$ can be uniquely expressed as 
\[(k_1,\ldots,k_{t-2},i,k_t+1-i,k_{t+1},\ldots,k_n)\]
by $(k_1,\ldots,k_n)\in I(t)$ and $1\leq i\leq k_t$. 
Therefore, by the induction hypothesis for $t-1$, we have 
\begin{align*}
&\sum_{r=l}^tB_{l,r}(\bp)\sum_{\bk\in I(r)}C_{l,r}(\bp,\bk)\\
&=B_{l,t-1}(\bp)\sum_{\bk\in I(t-1)}
\frac{1}{m^{k_{t-1}}}C_{l,t}(\bp,\bk)
+B_{l,t}(\bp)\sum_{\bk\in I(t)}C_{l,t}(\bp,\bk)\\
&=B_{l,t}(\bp)\sum_{\bk\in I(t)}
\Biggl(\sum_{i=1}^{k_t}
\frac{1}{m^i}\frac{p_l+\cdots+p_{t-1}}{(m+p_l+\cdots+p_{t-1})^{k_t+1-i}}\\
&\hspace{100pt}
+\frac{1}{(m+p_l+\cdots+p_{t-1})^{k_t}}\Biggr)
C_{l,t+1}(\bp,\bk). 
\end{align*}
Now, from the equality 
\[\sum_{i=1}^K\frac{1}{x^i}\frac{y}{(x+y)^{K+1-i}}+\frac{1}{(x+y)^K}
=\frac{1}{x^K}, \]
our claim \eqref{eq:r<n} follows immediately. 
\end{proof}

By Lemma \ref{lem:r<n} for $t=n-1$, $D_l(\bp)$ can be written as 
\[D_l(\bp)=B_{l,n-1}(\bp)\sum_{\bk\in I(n-1)}
\frac{1}{m^{k_{n-1}}}C_{l,n}(\bp,\bk)
+B_{l,n}(\bp)\sum_{\bk\in I(n)}C_{l,n}(\bp,\bk). \]
By definition, we see that 
\[C_{l,n}(\bp,\bk)=\frac{1}{(m+p_l+\cdots+p_{n-1})^{k_n}}\]
and 
\begin{align*}
I(n-1)&=\bigl\{(1,\ldots,1,i,k-(n-2)-i)\bigm| 1\leq i\leq k-n\bigr\}, \\
I(n)&=\bigl\{(1,\ldots,1,1,k-(n-1))\bigr\}. 
\end{align*}
Hence a computation similar to the proof of Lemma \ref{lem:r<n} shows that 
\begin{align}
D_l(\bp)&=B_{l,n}(\bp)\Biggl(\sum_{i=1}^{k-n}\frac{1}{m^i}
\frac{p_l+\cdots+p_{n-1}}{(m+p_l+\cdots+p_{n-1})^{k-(n-2)-i}}\notag\\
&\hspace{120pt}
+\frac{1}{(m+p_l+\cdots+p_{n-1})^{k-(n-1)}}\Biggr)\notag\\
&=B_{l,n}(\bp)\frac{1}{m^{k-n}}\frac{1}{m+p_l+\cdots+p_{n-1}}\notag\\
&=B_{l,n-1}(\bp)\frac{1}{m^{k-n}}
\frac{1}{(p_l+\cdots+p_{n-1})(m+p_l+\cdots+p_{n-1})}\notag\\
&=B_{l,n-1}(\bp)\frac{1}{m^{k-n+1}}
\biggl(\frac{1}{p_l+\cdots+p_{n-1}}-\frac{1}{m+p_l+\cdots+p_{n-1}}\biggr).
\label{eq:D_l(p)}
\end{align}

\begin{prop}\label{prop:fix l}
For $l=1,\ldots,n$, 
\[\sum_{\lvert J\rvert=l}S(m,J)
=\frac{1}{m^{k-(n-1)}}
\sum_{\substack{m>q_1>\cdots>q_{l-1}\geq 1\\ 
1\leq q_l\leq\cdots\leq q_{n-1}\leq m}}\frac{1}{q_1q_2\cdots q_{n-1}}. \]
\end{prop}
\begin{proof}
First, it is obvious from the definition that 
\[A_l=\sum_{m>q_1>\cdots>q_{l-1}\geq 1}\frac{1}{q_1q_2\cdots q_{l-1}}. \]
Hence, by \eqref{eq:|J|=l} and \eqref{eq:D_l(p)}, 
it suffices to show 
\begin{multline}\label{eq:fix l'}
\sum_{\bp\in\N^{n-l}}B_{l,n-1}(\bp)
\biggl(\frac{1}{p_l+\cdots+p_{n-1}}-\frac{1}{m+p_l+\cdots+p_{n-1}}\biggr)\\
=\sum_{1\leq q_l\leq\cdots\leq q_{n-1}\leq m}
\frac{1}{q_lq_{l+1}\cdots q_{n-1}}. 
\end{multline}
This is done by computing the summation for $p_{n-1},p_{n-2},\ldots,p_l$ 
successively. In fact, the first summation is 
\begin{align*}
&\sum_{p_{n-1}=1}^\infty B_{l,n-1}(\bp)
\biggl(\frac{1}{p_l+\cdots+p_{n-1}}-\frac{1}{m+p_l+\cdots+p_{n-1}}\biggr)\\
&=\sum_{p_{n-1}=1}^m B_{l,n-1}(\bp)\frac{1}{p_l+\cdots+p_{n-1}}\\
&=\sum_{p_{n-1}=1}^m B_{l,n-2}(\bp)\frac{1}{p_{n-1}}
\biggl(\frac{1}{p_l+\cdots+p_{n-2}}-\frac{1}{p_l+\cdots+p_{n-1}}\biggr). 
\end{align*}
Then the second summation is 
\begin{align*}
&\sum_{p_{n-2}=1}^\infty\sum_{p_{n-1}=1}^m B_{l,n-2}(\bp)\frac{1}{p_{n-1}}
\biggl(\frac{1}{p_l+\cdots+p_{n-2}}-\frac{1}{p_l+\cdots+p_{n-1}}\biggr)\\
&=\sum_{p_{n-1}=1}^m\sum_{p_{n-2}=1}^{p_{n-1}}
B_{l,n-2}(\bp)\frac{1}{p_{n-1}}\frac{1}{p_l+\cdots+p_{n-2}}\\
&=\sum_{p_{n-1}=1}^m\sum_{p_{n-2}=1}^{p_{n-1}}
B_{l,n-3}(\bp)\frac{1}{p_{n-2}p_{n-1}}
\biggl(\frac{1}{p_l+\cdots+p_{n-3}}-\frac{1}{p_l+\cdots+p_{n-2}}\biggr)
\end{align*}
and so on. The last summation is 
\begin{multline*}
\sum_{p_l=1}^\infty\sum_{1\leq p_{l+1}\leq\cdots\leq p_{n-1}\leq m}
B_{l,l}(\bp)\frac{1}{p_{l+1}\cdots p_{n-1}}
\biggl(\frac{1}{p_l}-\frac{1}{p_l+p_{l+1}}\biggr)\\
=\sum_{1\leq p_l\leq\cdots\leq p_{n-1}\leq m}
\frac{1}{p_lp_{l+1}\cdots p_{n-1}}, 
\end{multline*}
which is the right-hand side of \eqref{eq:fix l'}. 
This proves Proposition \ref{prop:fix l}. 
\end{proof}

\section{Third step}
For $l=1,\ldots,n$, put 
\[Q_l=\biggl\{\bq=(q_1,\ldots,q_{n-1})\in\N^{n-1}\biggm|
\begin{array}{l}
m>q_1>\cdots>q_{l-1}\geq 1,\\ 
1\leq q_l\leq\cdots\leq q_{n-1}\leq m
\end{array}\biggr\}. \]
Then Proposition \ref{prop:fix l} says that 
\begin{equation}\label{eq:Q_l}
\sum_{\lvert J\rvert=l}S(m,J)
=\frac{1}{m^{k-(n-1)}}\sum_{\bq\in Q_l}\frac{1}{q_1q_2\cdots q_{n-1}}.
\end{equation}
On the other hand, it is easy to see that 
\[Q_1=\bigl\{(m,\ldots,m)\}\amalg (Q_1\cap Q_2) \]
and 
\[Q_l=(Q_l\cap Q_{l-1})\amalg (Q_l\cap Q_{l+1})\qquad (2\leq l\leq n), \]
where we put $Q_{n+1}=\emptyset$. Hence the inclusion-exclusion argument 
implies 
\begin{equation}\label{eq:sum Q_l}
\frac{1}{m^{n-1}}=\sum_{l=1}^n(-1)^{l-1}
\sum_{\bq\in Q_l}\frac{1}{q_1q_2\cdots q_{n-1}}. 
\end{equation}
Combining \eqref{eq:Q_l} and \eqref{eq:sum Q_l}, we obtain 
\begin{align*}
\sum_{\emptyset\ne J\subset\{1,\ldots,n\}}(-1)^{\lvert J\rvert-1}S(m,J)
&=\sum_{l=1}^n(-1)^{l-1}\sum_{\lvert J\rvert=l}S(m,J)\\
&=\frac{1}{m^{k-(n-1)}}\sum_{l=1}^n(-1)^{l-1}
\sum_{\bq\in Q_l}\frac{1}{q_1q_2\cdots q_{n-1}}\\
&=\frac{1}{m^k}. 
\end{align*}
Now the proof of Proposition \ref{prop:MainThm3} is complete. 

\bigbreak

\proof[Acknowledgements]
The author expresses his gratitude to Hirofumi Tsumura 
for introducing the conjecture of \cite{EMT10} in his talk.

\end{document}